\documentclass[11pt]{article}

\usepackage{latexsym,color,amsmath,amsthm,amssymb,amscd,amsfonts}

\usepackage{graphicx}
\usepackage{cancel}
\usepackage{stmaryrd}
\usepackage{commath}
\usepackage{mathtools}
\usepackage{amssymb}
\usepackage{enumerate}
\usepackage{bbm}
\usepackage{hyperref,mathtools}
\usepackage[title]{appendix}

\newtheorem{conj}{Conjecture}[section]
\newtheorem{theorem}[conj]{Theorem}

\newtheorem{remark}[conj]{Remark}
\newtheorem{lemma}[conj]{Lemma}
\newtheorem{proposition}[conj]{Proposition}

\newtheorem{defn}[conj]{Definition}

\newcommand\independent{\protect\mathpalette{\protect\independent}{\perp}} 
\def\independent#1#2{\mathrel{\rlap{$#1#2$}\mkern2mu{#1#2}}}

\newcommand\restr[2]{{
  \left.\kern-\nulldelimiterspace 
  #1 
  \littletaller 
  \right|_{#2} 
  }}

\newcommand{\littletaller}{\mathchoice{\vphantom{\big|}}{}{}{}}

\renewcommand{\P}{\mathbb{P}}

\DeclareMathOperator{\Var}{Var}

\newcommand{\Z}{\mathcal{Z}}

\renewcommand{\Z}{\mathbb{Z}}

\date{}

\author{Heshan Aravinda}

\title{Entropy-variance inequalities for discrete log-concave random variables via degree of freedom}

\begin{document}

\maketitle

\begin{abstract}
We utilize a discrete version of the notion of degree of freedom to prove a sharp min-entropy-variance inequality for integer valued log-concave random variables. More specifically, we show that the geometric distribution minimizes the min-entropy within the class of log-concave probability sequences with fixed variance. As an application, we obtain a discrete R\'{e}nyi entropy power inequality in the log-concave case, which improves a result of Bobkov, Marsiglietti, and Melbourne (2022).

\end{abstract}


\section{Introduction}
\label{intro}
Recall an integer valued random variable $X$ is said to be log-concave if its probability mass function $p$ satisfies $$p^2(z)\geq p(z-1)p(z+1)\,\,\,\,\text{for all $z \in \mathbb{Z}$}$$ and the support of $X$ is contiguous. The log-concave assumption provides a broad, yet natural, convolution-stable class of distributions
on $\mathbb{Z}.$ Examples include Bernoulli, sums of independent Bernoulli, geometric, binomial, negative binomial and Poisson. We refer to \cite{SW,S,JG,NO,BJRP} for more details about discrete log-concavity. In recent years, there has been a great deal of interest in studying these distributions and producing results, in analogy with log-concave distributions in the continuous setting \cite{KL,JKM,MM2,MP}. The purpose of this article is to investigate information-theoretic properties of such random variables using a sophisticated localization-type argument. More specifically, we adapt the notion of degree of freedom to explore sharp entropy-variance inequalities in the log-concave setting.  \ \vskip1mm  

The notion of degree of freedom of a log-concave function was first introduced in \cite{FG} by Fradelizi and Gu\'{e}don. Therein, the authors used it to develop a localization technique, which provides a powerful tool for proving geometric and functional inequalities for log-concave measures in $\mathbb{R}^n$. A discrete analog of localization has been developed in \cite{MarMel}. Applications include dilation inequalities, a discrete Pr\'{e}kopa-Leindler, concentration for ultra log-concave distributions \cite{HesMarMel} and a proof of a  strengthened version of a conjecture of Feige \cite{AAMM}. Broadly speaking, these problems were treated as optimization problems with one constraint where the reduction of proofs to a certain type of distribution was possible due to the identification of extreme points and the use of the Krein-Milman theorem. However, one needs a more general approach to deal with problems involving more than one constraint. In such cases, discrete degree of freedom can be leveraged. In particular, this notion can be used to describe extreme points of a subset of integer valued log-concave probabilities satisfying more than one constraint. \ \vskip1mm

For completeness, let us first recall the notion of degree of freedom of a log-concave function on $\mathbb{R}$.

\begin{defn}[\cite{FG}]
    \label{dof_continuous}
    Let $U: \mathbb{R} \to \mathbb{R} \cup \{+\infty\}$ be convex, denote by $E$ the domain
of $U$ (i.e. $E = \{x \colon U(x)<+\infty\}$. The degree of freedom of $e^{-U}$ is the largest integer $k$ such that there exist $\alpha>0$ and linearly independent continuous functions $U_1,U_2,\dots,U_k$ defined on $E$ such that for all $\epsilon_1,\epsilon_2,\dots,\epsilon_k \in [-\alpha,\alpha]\,,$ the function   $e^{-U} (1 + \sum_{i=1}^k \epsilon_i U_i)$ is log-concave.
\end{defn}
The idea is to adapt this definition to log-concave functions on integers. Let us begin with the notion of discrete convexity. A function $V:\mathbb{Z} \to \mathbb{R} \cup \{+ \infty\}$ is said to be convex if $$\Delta^2 V(z):= V(z-1)- 2V(z)+V(z+1)\geq 0\,\,\,\,\text{for all $z \in \mathbb{Z}$.}$$

Equivalently, $V$ is convex on $\mathbb{Z}$ if and only if there exists a continuous and convex\footnote{Let $E \subseteq \mathbb{R}$ be convex. A function $U \colon E \to \mathbb{R}$ is convex iff for all $0\leq t\leq 1$ and $x,y \in E$, $U(tx + (1-t)y)\leq tU(x) + (1-t)U(y)$. }function $\Bar{V}$ such that $\Bar{V} = V$ on $\mathbb{Z}$. A function $V$ is concave on $\mathbb{Z}$ if and only if $-V$ is convex on $\mathbb{Z}$. A function is affine (discrete) if it is both discrete convex and concave. If the support of $V$ only contains two points, say $\{z-1,z\}$ for some $z \in \mathbb{Z}$, then $V$ is necessarily affine, since $\Bar{V}$ is affine on $[z-1,z]$. A function $f$ is said to be log-concave if $\log(f)$ is concave, or equivalently, if $f^2 (z)\geq f(z-1)f(z+1)$ for all $z \in \mathbb{Z}$. Similarly, one can define the log-convexity of functions defined on integers.\ \vskip1mm

We now define the notion of degree of freedom in the discrete setting. For our purpose, it suffices to consider compactly supported log-concave functions. Let $a,b \in \mathbb{Z}$ with $a<b$. Denote by $\llbracket a,b \rrbracket$, an interval of integer points from $a$ to $b$. $[a,b] \subset \mathbb{R}$ is the smallest closed convex set containing $\llbracket a,b \rrbracket$. Suppose $f$ is log-concave supported on $\llbracket a,b \rrbracket$, i.e. $f$ is of the form $e^{-V}$, where $V$ is convex and $V(z)< \infty$ on $\llbracket a,b \rrbracket$. 
 \begin{defn}
 \label{deg. of freedom}
The degree of freedom of $e^{-V}$ is the largest $k$ such that there exist $\alpha>0$ and linear independent functions $W_1,W_2,\dots,W_k$ defined on $\llbracket a,b \rrbracket$ such that for all $\epsilon_1,\epsilon_2,\dots,\epsilon_k \in [-\alpha,\alpha]\,,$ the function   $e^{-V} (1 + \sum_{i=1}^k \epsilon_i W_i)$ is log-concave on $\llbracket a,b \rrbracket$.
 \end{defn}
 Our first main result (see also \cite{NS}\footnote{In their recent work, Jakimiuk et al. have independently developed results that overlap with Theorem 1.3.}) is as follows. One may view this as a discrete analog of \cite[Proposition 2]{FG}, where the authors have established a necessary and sufficient condition for the finiteness of degree of freedom of a continuous log-concave function.
\begin{theorem}
\label{necessary condition}
If the degree of freedom of $e^{-V}$ is $k+1$, then there exist $k$ number of affine functions $\psi_1, \psi_2,\dots,\psi_k$ on $\llbracket a,b \rrbracket$ such that $V = \displaystyle \max_{i=1}^k \psi_i.$
\end{theorem}
The proof of the theorem relies on the continuous case. In fact, we prove a stronger result. Namely, we will show that the degrees of freedom of $e^{-V}$ and its continuous linear extension are equal (Lemma 2.1). This result allows one to determine the degree of freedom of a given discrete log-concave function, whereas Theorem \ref{necessary condition} only provides a necessary condition for degree of freedom to be finite. \ \vskip1mm 
Let us recall \cite[Proposition 2]{FG}, which states that the degree of freedom of a continuous log-concave function $e^{-U}$ is $k+1$ if and only if there exist $k$ (but not less than $k$) continuous affine functions $\phi_1,\phi_2,\dots,\phi_k$ such that $U = \max\{\phi_1, \phi_2,\dots,\phi_k\}$. In general, this is false in the discrete setting, as can be seen by taking $f=e^{-V}$, where $V = \{3,2, \frac{4}{3}, 2\}$ defined on $\llbracket 0,3 \rrbracket$. The log-concavity of $f$ follows directly from the definition. In order to determine the degree of freedom, let us first construct $\bar{V}$ by extending $V$ linearly on each sub-interval $[z-1,z]$ for $z=1,2,3$. $\bar{V}$ is as follows.
\[\bar{V}(x) = \begin{cases} 
      3-x  &  x \in [0,1] \\
      
     \frac{8}{3}  -\frac{2}{3}x &  x \in [1,2]\\
      
      \frac{2}{3}x  &  x \in [2,3]
   \end{cases}
\] 
Clearly, $\bar{V}$ is convex on $[0,3]$. It follows from \cite[Proposition 2]{FG} that the degree of freedom of $e^{-\bar{V}}$ is $4$. By Lemma \ref{discrete and continuous deg. of freedom} (to be proved), $e^{-V}$ has $4$ degrees of freedom. However, $V$ can be written as the maximum of $\psi_1(z) = 3-z$ and $\psi_2(z) = \frac{2}{3}z$ defined on $\llbracket 0,3 \rrbracket$.\ \vskip1mm
As mentioned, the main goal of this paper is to investigate the relationship between entropy and variance within the class of discrete log-concave probabilities. In particular, it is interesting to understand entropy minimizers among these random variables when variance is fixed. Such problems have been considered in the continuous setting \cite{MNT,BN,BM,MK}. However, the work in the discrete setting is limited. First, let us recall the following: given an integer valued random variable $X$ with the mass function $p$,  the R\'{e}nyi entropy of order $\alpha>0$, $\alpha \neq 1$, is defined by,  $$ H_{\alpha}(X) =\dfrac{-1}{\alpha-1}\log \displaystyle \sum_{z \in \mathbb{Z}}\,p(z)^{\alpha}.$$ If $ \alpha \to \infty$, one obtains the min-entropy $H_{\infty}(X) = - \log \max_z\, p(z)$. The limiting case $\alpha \to 1$ recovers the Shannon entropy $H(X) =  H_1(X) = - \sum_{z \in \mathbb{Z}} p(z)\log p(z)$. The R\'{e}nyi entropy power of order $\alpha$ is, $$N_{\alpha}(X) = e^{2H_{\alpha}(X)} = \left( \displaystyle \sum_{z \in \mathbb{Z}} p(z)^{\alpha}\right)^{\frac{-2}{\alpha-1}}.$$
Similarly, one can define the R\'{e}nyi entropy for absolutely continuous random variables with densities with respect to the Lebesgue measure.\
\vskip1mm
Entropy inequalities are a central topic of study in information theory. They have found striking applications in combinatorics \cite{K,MMT,B}. We refer the reader to \cite{MT,J,G} for recent works on discrete R\'{e}nyi entropy inequalities for log-concave random variables and their variants.\ \vskip1mm
Our next main result is motivated by the work of Bobkov-Marsiglietti-Melbourne \cite{BobMarMel}, in which they have utilized a majorization technique to establish variance bounds for min-entropy. Namely, it has been shown that $N_{\infty}(X) \geq \frac{1}{4} + \mathrm{Var}(X)$ for $X$- discrete log concave (see \cite[Theorem 1.1]{BobMarMel}). We will show the following:  
\begin{theorem}
\label{min-entropy}
If $X$ is a discrete log-concave random variable, then \begin{equation}
\label{b1}
 N_{\infty}(X) \geq 1 + \mathrm{Var}(X).
\end{equation}
This bound is asymptotically attained for geometric distributions.
\end{theorem}
For a geometric distribution with the probability mass function $p(z) = \theta(1-\theta)^z,\, z \in \{0,1,2,\dots\}$, we have \begin{equation*}
\frac{1 + \mathrm{Var}(X)}{N_{\infty}(X)} = 1-\theta+\theta^2,
\end{equation*}
which goes to $1$ as $\theta \to 0$ or $\theta \to 1.$\ \vskip1mm
Since $\alpha \to N_{\alpha}(X)$ is non-increasing, one may obtain an inequality for the Shannon entropy power, i.e. $N(X)\geq 1 + \mathrm{Var}(X)$. However, this inequality may be improved. It would be interesting to determine the optimal value of $c$ for which $N(X)\geq 1 + c\mathrm{Var}(X)$ holds.  \ \vskip1mm
In the Euclidean setting, the entropy power inequality (EPI) states that the entropy power of the sum of independent random vectors is at least the sum of their entropy powers; equality holds if and only if the random vectors are Gaussian with proportional covariance matrices. A more general inequality also holds for the R\'{e}nyi entropy of order $\alpha$ up to some $\alpha$ - dependent factor \cite{BC}. In general, the discrete analog of the EPI does not hold. Nonetheless, there have been numerous attempts to formulate  discrete versions of the EPI. See the recent works in this direction \cite{HAT,HV,MWJ}.
\vskip1mm
Theorem \ref{min-entropy} implies a discrete entropy power inequality for log-concave random variables, providing a partial answer to a question left open in \cite{BobMarMel}. As mentioned therein, we consider the quantity $\Delta_{\alpha}(X) = N_{\alpha}(X)-1$ to be the discrete analog of the usual R\'{e}nyi entropy power. 
\begin{theorem}
\label{EPI}
Let $S_n = \sum_{i=1}^n X_i$, where $\{X_i\}_{1\leq i \leq n}$ is a collection of independent discrete log-concave random variables. Then,
\begin{equation}
\label{bo1}
    \Delta_{\alpha} (S_n) \geq c_{\alpha} \displaystyle \sum_{i=1}^n \Delta_{\alpha} (X_i)
\end{equation}
for $\alpha>1$ with $c_{\alpha} = \dfrac{\alpha -1}{{4(3 \alpha -1)}}.$
\end{theorem}
In the regime $1< \alpha \leq 2$, the constant $c_{\alpha}$ can be improved to $\frac{\alpha-1}{3 \alpha-1}$ \cite{BobMarMel}. It has also been shown that the inequality (\ref{bo1}) holds for the Poisson-Bernoulli random variables and $\alpha \geq 2$ with $c_{\alpha} = \frac{\alpha}{6 (\alpha-1)}$ \cite{MMR}. 
 \vskip1mm
Let us outline the paper. In Section \ref{S2}, we prove a key lemma that unveils the relationship between the degrees of freedom of a compactly supported discrete log-concave function and its continuous counterpart. As an immediate consequence, we get Theorem \ref{necessary condition}. Section \ref{S3} contains entropy inequalities. We show that Theorem \ref{min-entropy} can be reduced to a certain class of random variables, which we identify using Theorem \ref{necessary condition}. It turns out their mass functions are necessarily log-concave with `V-shaped' potentials. The proof of the inequality for such distributions is included in Section \ref{S4}.
\section*{Acknowledgements} I would like to thank Arnaud Marsiglietti for numerous helpful discussions. I would also like to thank the referees for
their valuable comments.
\section{Degree of freedom in the discrete setting}
\label{S2}
We start with the following lemma.
\begin{lemma} 
\label{discrete and continuous deg. of freedom} 
Let $a,b \in \Z$ such that $a \leq b$. Let $V \colon \llbracket a,b \rrbracket \to \mathbb{R}$ be convex. Then, 
\begin{center}
    \text{Degree of freedom of $e^{-V}=$} \text{Degree of freedom of $e^{-\Bar{V}}$,}
\end{center}
where $\bar{V}$ is the continuous function obtained by extending $V$ linearly on each sub-interval $[z-1,z] \subset [a,b]$  , $z \in \Z$.
\end{lemma}
Note that $\Bar{V}$ in Lemma \ref{discrete and continuous deg. of freedom} is convex on $[a,b].$ Indeed, since $\Delta^2 V(z)\geq 0$ on $\llbracket a,b \rrbracket$ and  $\Bar{V}^\prime(x)= V(z)- V(z-1)$ on each $(z-1,z)$, i.e. $\Bar{V}^\prime$ is non-decreasing. Therefore, the degree of freedom of $e^{-\Bar{V}}$ is defined. In fact, it is finite, whereas the degree of freedom of a compactly supported continuous log-concave function can be infinite, as can be seen by taking $g(x) =e^{-x^2}$ on $[-1,1]$. Let us now give a proof of Theorem \ref{necessary condition}.

\begin{proof}[Proof of Theorem \ref{necessary condition}] Let $k+1$ be the degree of freedom of $e^{-V}$. Equivalently, by Lemma 2.1, $e^{-\Bar{V}}$ has $k+1$ degrees of freedom. By \cite[Proposition 2]{FG}, there exist $k$ affine functions $\Bar{\psi_1},\Bar{\psi_2},\dots,\Bar{\psi_k}$ on $[a,b]$ such that $\Bar{V} = \max\{ \Bar{\psi_1}, \Bar{\psi_2},\dots,\Bar{\psi_k}\}$. Restricting $\Bar{V}$ to $\llbracket a,b \rrbracket$, we get $V = \max\{\psi_1,\psi_2,\dots,\psi_k\}$, where $\psi_i = $
$\restr{\Bar{\psi_i}}{ \llbracket a,b \rrbracket }$, the restriction of each $\Bar{\psi_i}$ to $ \llbracket a,b \rrbracket$. 
\end{proof}
It remains to prove Lemma \ref{discrete and continuous deg. of freedom}. As in the continuous case, we begin with the following important fact. \vskip2mm 
\noindent
\textbf{FACT I:} Suppose $\alpha,k$ and $W_1,W_2,\dots,W_k$ satisfy Definition \ref{deg. of freedom}. If $V$ is affine on $\llbracket a,b \rrbracket$, then each $W_i$ is affine on $\llbracket a,b \rrbracket$. \ \vskip1mm
We simply adapt the argument used in \cite{FG} to our setting. The result is trivial if $b-a=1$. Let us assume that $b-a\geq 2$. Fix $i \in \{1, \dots, k\}$ and let $\epsilon_j=0$ for all $j \neq i$ in Definition \ref{deg. of freedom}. Then, for all $\epsilon \in [-\alpha,\alpha]$, the function $e^{-V} (1 + \epsilon W_i)$ is log-concave on $\llbracket a,b \rrbracket$. Equivalently, $V - \log(1 + \epsilon W_i)$ is convex. This implies $\log(1 + \epsilon W_i)$ is concave, i.e. $\Delta^2 h_{\epsilon}(z)\leq 0$ for all $z \in  \llbracket a,b \rrbracket$, where $h_{\epsilon}(z) = \log(1 + \epsilon W_i(z))$. Consider the two cases:  $\frac{\Delta^2 h_{\epsilon}(z)}{\epsilon} \leq 0$ when $\epsilon>0$ and $\frac{\Delta^2 h_{\epsilon}(z)}{\epsilon} \geq 0$ when $\epsilon<0$. Since  $\lim_{\epsilon \to 0}\, \frac{h_{\epsilon}(z)}{\epsilon} = W_i (z)$, we conclude by taking the limit $\epsilon \to 0$ in $\frac{\Delta^2 h_{\epsilon}(z)}{\epsilon}$. \ \vskip1mm
The following fact verifies Lemma \ref{discrete and continuous deg. of freedom} for log-affine functions $e^{-V}$ (equivalently when $V$ is affine).\vskip2mm
\noindent
\textbf{FACT II:} Let $m,n \in \mathbb{R}$. Let $V(z) = mz+n$ be defined on $\llbracket a,b \rrbracket$. Then, the degree of freedom of $e^{-V}$ is $2$. \
\vskip1mm
Without loss of generality, assume that $V$ is non-constant affine. Choose $W_1(z) = \mathbbm{1}_{\llbracket a,b \rrbracket} $ and $W_2(z) = V(z) $. Let $\alpha = \min(\frac{1}{2},(2\max(|W_1| +|W_2|))^{-1}).$ Then, for each $(\epsilon_1,\epsilon_2) \in [-\alpha,\alpha]^2$, we have $1+ \epsilon_1W_1 + \epsilon_2W_2 \geq \frac{1}{2}$, i.e. $\log(1+ \epsilon_1W_1(z) + \epsilon_2W_2(z))$ is defined for all $z \in \llbracket a,b \rrbracket$. Moreover, $W_1$ and $W_2$ are linear independent. Let us check the log-concavity condition. For convenience, let $\mu = e^{-V}(1+ \epsilon_1W_1 + \epsilon_2W_2)$. The result is trivial if $b-a=1$. Assume $b-a\geq 2$. Then, $$\mu^2(z) - \mu(z-1)\mu(z+1) = e^{-2(mz+n)} (\epsilon_2 m)^2,$$ which is non-negative for all $\abs{\epsilon_i}\leq \alpha$ and $z \in \llbracket a,b \rrbracket$. This proves that $e^{-V}$ has at least $2$ degrees of freedom. In fact, the degree of freedom cannot exceed $2$, since our choice for linear independent functions satisfying Definition \ref{deg. of freedom}, is necessarily affine by \textbf{Fact I}.\ 
\begin{remark}
Note that if $e^{-V}$ is supported at a point, i.e. $e^{-V} = c\mathbbm{1}_{\{a\}}$ for some $c>0$, with $V(z)= -\log c$ at $z=a$ and $\infty$ otherwise, then we identify the degree of freedom of $e^{-V}$ as $1$. Indeed, since one can choose $W_1=V$ and $\alpha = \min(\frac{1}{2}, \frac{1}{2|V|})$ satisfying Definition \ref{deg. of freedom} (with the convention $e^{-V}V=0$ for $V(z) = \infty$). 
\end{remark}
Let us now prove Lemma \ref{discrete and continuous deg. of freedom}. We shall assume that $V$ is not affine, i.e. $b-a \geq 2$. 
\begin{proof}[Proof of Lemma \ref{discrete and continuous deg. of freedom}]
By construction, $\Bar{V}$ is piecewise linear with a finite number of parts, say $k (\geq 2)$, i.e. $\Bar{V}$ is the maximum of $k$ (and no less) affine functions defined on $[a,b]$. Then by \cite[Proposition 2]{FG}, the degree of freedom of $e^{-\Bar{V}}$ is $k+1$. We will show that the degree of freedom of $e^{-V}$ is $k+1$. Let $z_0=a$ and $z_{k} = b$. There exist points $a<z_1<z_2<...<z_{k-1}<b$ such that $\bar{V}$ is affine on $[z_{i-1},\, z_i]\,$ for $i=1,2,\dots,k$. First, we show that the degree of freedom of $e^{-V}$ is \textbf{\textit{at least $k+1$}}. The idea is to use the construction in \cite{FG} to obtain $k+1$ linear independent functions satisfying Definition \ref{deg. of freedom}.\ \vskip1mm
Due to convexity of $\bar{V}$, one can find  $x_1<x_2<...<x_k$ such that
\begin{itemize}
    \item $z_{i-1}<x_i<z_i$ for $i=1,2,\dots,k$. 
    \item $\bar{V}^\prime (x_i) < \bar{V}^\prime(x_{i+1})$ for $i=1,2,\dots,k-1$.
\end{itemize}
For each $i=1,2,\dots,k$, define \[\Bar{W}_i(x) = \begin{cases} 
            \Bar{V}(x) &  x < x_i\\
      
      \Bar{V}(x_i) + (x-x_i)\Bar{V}^\prime (x_i)  &  x \geq x_i
   \end{cases}
\] and $\bar{W}_0 = \mathbbm{1}_{[a,b]}\,$ (let us assume $\Bar{V}(x_1) \neq 0$ without loss of generality, since otherwise one can reflect the whole picture about $\frac{1}{2}(a + b)$. Equivalently
one can do the construction ``from the right"). Denote by $W_i$, the restriction of $\Bar{W}_i$ to $\llbracket a,b \rrbracket$, i.e. for $z \in \llbracket a,b \rrbracket$
\[W_i(z) = \begin{cases} 
            V(z) &  z < x_i\\
      
      \Bar{V}(x_i) + (z-x_i)\Bar{V}^\prime (x_i)  &  z\geq x_i
   \end{cases}
 \] and $W_0 = \mathbbm{1}_{\llbracket a,b \rrbracket}$. \ \vskip1mm
 We claim that $\{W_i\}_{i=0}^k$ satisfies Definition \ref{deg. of freedom}.  First, we show the log-concavity condition. It follows from \cite[Proposition 2]{FG} that for a sufficiently small $\alpha>0$, the function $e^{-\Bar{V}}(1+ \sum_{i=0}^k \epsilon_i \Bar{W}_i)$ is log-concave on $[a,b]$. Since $e^{-\Bar{V}}(1+ \sum_{i=0}^k \epsilon_i \Bar{W}_i) = e^{-V}(1+ \sum_{i=0}^k \epsilon_i W_i)$ on $\llbracket a,b \rrbracket$, we conclude that for all $\epsilon_i \in [-\alpha,\alpha]$, $e^{-V}(1+ \sum_{i=0}^k \epsilon_i W_i)$ is log-concave on $\llbracket a,b \rrbracket$. \ \vskip1mm
Next, we verify the linear independence of $\{W_i\}_{i=0}^k$. Let $h(z) = \sum_{i=0}^k c_i W_i(z)$. We need to show that $h(z) = 0$ for all $z \in \llbracket a,b \rrbracket$ implies $c_i=0,\,i=0,1,2,...k.$ Let us first show that $c_{k}=0$. Letting $z=z_k, z_{k-1}$ and $z=z_{k-2}$ in $h(z)=0$, respectively, we obtain the following linear equations.
\begin{align*}
    L_1\coloneqq c_0 + \sum_{i=1}^k c_i \left(\Bar{V}(x_i) + (z_{k}-x_i)\Bar{V}^\prime (x_i)\right) &= 0.\\
 L_2 \coloneqq c_0 + \sum_{i=1}^{k-1} c_i \left(\Bar{V}(x_i) + (z_{k-1}-x_i)\Bar{V}^\prime (x_i)\right) + c_k V(z_{k-1}) &= 0.\\
   L_3\coloneqq c_0 + \sum_{i=1}^{k-2} c_i \left(\Bar{V}(x_i) + (z_{k-2}-x_i)\Bar{V}^\prime (x_i)\right) + (c_{k-1}+c_k) V(z_{k-2})  &= 0.
\end{align*}
$L_1- L_2 =0$ and dividing by $z_k - z_{k-1}$ yield $\sum_{i=1}^{k-1} c_i \Bar{V}^\prime (x_i)+ c_k\frac{V(z_k)-V(z_{k-1})}{z_k - z_{k-1}}=0$. Equivalently, $\sum_{i=1}^k c_i \Bar{V}^\prime (x_i) = 0$, since $\Bar{V}$ is affine on $[z_k - z_{k-1}]$ with derivative $\Bar{V}^\prime (x_k)$. Similarly, $\sum_{i=1}^{k-1} c_i \Bar{V}^\prime (x_i) + c_k \Bar{V}^\prime (x_{k-1}) = 0$ follows from $L_2-L_3 = 0$, and dividing by $z_{k-1}-z_{k-2}$. Then, the result follows by subtracting these two equations. \vskip1mm One may proceed similarly to show $c_i=0$ for $i=1,2,\dots,k-1$. After plugging $z=z_{k-3},z_{k-4},\dots,z_0$ into $h(z)=0$, simplifying as before and combining with the case $c_k = 0$, one obtains the following system of $k-1$ linear equations.
\begin{align*}
 M_1 \coloneqq \sum_{i=1}^{k-1} c_i \Bar{V}^\prime (x_i) &=0.\\
M_2 \coloneqq\sum_{i=1}^{k-2} c_i \Bar{V}^\prime (x_i) + c_{k-1} \Bar{V}^\prime (x_{k-2}) &=0.\\
M_3 \coloneqq\sum_{i=1}^{k-3} c_i \Bar{V}^\prime (x_i) + (c_{k-2}+ c_{k-1}) \Bar{V}^\prime (x_{k-3}) &=0.\\
      .\\
      .\\
       .\\
 M_{k-1} \coloneqq \sum_{i=1}^{k-1} c_i &= 0.
\end{align*}
$M_1 - M_2 = 0 \implies c_{k-1}=0,\, M_2-M_3 = 0 \implies c_{k-2} = 0$ and so on. \ \vskip1mm
Let us now prove the reverse direction, i.e. the degree of freedom of $e^{-V}$ is \textbf{\textit{ at most $k+1$}}. The argument is analogous to the continuous case. For completeness, we explain the details below.\ \vskip1mm
Suppose on contrary that the degree of freedom of $e^{-V}$ is more than $k+1$. Then, there exist $\alpha>0$ and $k+2$ linear independent functions $V_1,V_2,\dots,V_{k+2}$ satisfying Definition \ref{deg. of freedom}. Since $\bar{V}$ is affine on each $[z_{i-1}, z_i],\,$ its restriction to $\llbracket a,b \rrbracket$, $V$ is also affine on each $\llbracket z_{i-1}, z_i \rrbracket.$ By \textbf{Fact I}, each $ V_j$ should also be affine on $\llbracket z_{i-1}, z_i \rrbracket$ for $i=1,2,\dots,k$. Let $C$ be the space of continuous functions defined on $\llbracket a,b \rrbracket$ that are affine on each $\llbracket z_{i-1}, z_i \rrbracket.$ We claim that $\dim(C) = k+1$. \ \vskip1mm
For each $i\in \{0,1,2,\dots,k\},\,$ define $f_i : \llbracket a,b \rrbracket \to \mathbb{R}$ such that
\begin{equation*}
f_i(z)= \begin{cases}
\dfrac{z-z_{i-1}}{z_i-z_{i-1}}& z \in \llbracket z_{i-1},z_i \rrbracket \\
1 - \dfrac{z-z_i}{z_{i+1}-z_i}& z \in \llbracket z_i,z_{i+1}\rrbracket\\
0& \mathrm{otherwise}
\end{cases}
\end{equation*}
 \vskip1mm
Clearly, $f_i \in C$ for all $i=0,1,\dots,k.$ By construction, $\{f_i\}_{i=0}^k$ is linear independent. Take any function $\ell \in C$. Then, its restriction to $\llbracket z_i, z_{i+1} \rrbracket$, denoted by $\Hat{\ell_i}$, is given by
$$\Hat{\ell_i}(z) \coloneqq \ell(z_i)+\frac{\ell(z_{i+1})-\ell(z_i)}{z_{i+1}-z_i}\,(z-z_i).$$
Choose $\beta_i = \ell(z_i)$ so that on $ \llbracket z_i, z_{i+1} \rrbracket, \Hat{\ell_i}= \beta_if_i+\beta_{i+1}f_{i+1}$ for all $i=0,1,2,\dots,k-1$. Hence, $\sum_{i=0}^k\beta_i f_i = \ell$. 
This implies $\{f_i\}_{i=0}^k$ spans $C$. Therefore, the collection $\{f_i\}_{i=0}^k$ forms a basis for $C$, proving the claim, and hence contradicting the supposition. 

\end{proof}
\section{Min-entropy-variance \& R\'enyi entropy inequalities}
\label{S3}
This section is devoted to the investigation of entropy inequalities. First, we prove Theorem \ref{EPI}. We need the following proposition from \cite{BobMarMel}.
\begin{proposition}[\cite{BobMarMel}]
\label{entropy bounds}
Let $1< \alpha \leq \infty$. If $X$ is any integer valued random variable with finite
variance, then $$1 \leq N_{\alpha}(X) \leq 1 + \dfrac{4(3 \alpha-1)}{\alpha -1} \mathrm{Var}(X).$$
\end{proposition}
\begin{proof}[Proof of Theorem \ref{EPI}] Let $S_n = X_1+ X_2 +\dots+X_n$, where $X_i$'s are independent discrete log-concave. Since log-concavity is preserved under independent summation, we apply inequality (\ref{b1}) of Theorem \ref{min-entropy} to $X = S_n$ to get the following:
$$\Delta_{\infty}(S_n) \geq \sum_{i=1}^n\mathrm{Var}(X_i).$$ 
The function $\alpha \to N_{\alpha}(X)$ is non-increasing. Therefore, one can extend the same lower bound to $\Delta_{\alpha}(S_n)$. We conclude by applying Proposition \ref{entropy bounds} to $X_i$, i.e. $$\Delta_{\alpha}(S_n) \geq \sum_{i=1}^n\mathrm{Var}(X_i) \geq \dfrac{\alpha -1}{4(3 \alpha-1)} \sum_{i=1}^n\Delta_{\alpha}(X_i). $$
\end{proof}

\begin{remark}\
\label{concentration function}
Given a discrete random variable $X$, its concentration function is defined by 
\begin{equation*}
Q(X;\lambda) = \displaystyle \sup_{z} \mathbb{P}\{z \leq X \leq z + \lambda\}, \,\,\, \lambda \geq 0.
\end{equation*}
 In particular, 
 \begin{equation*}
N^{-\frac{1}{2}}_{\infty}(X) = Q(X ; 0).
\end{equation*}
Theorem \ref{min-entropy}, together with \cite[Lemma 8.1]{BobMarMel}, implies the following for $X$- discrete log-concave random variables:
\begin{equation*}
    Q(X ;\lambda) \leq \frac{\lambda+1}{\sqrt{1+\frac{\lambda (\lambda+2)}{12} + \Var(X) }}.
\end{equation*}
This improves Proposition 8.3 from \cite{BobMarMel}. 
\end{remark}

We now proceed to the proof of Theorem \ref{min-entropy}. The idea is to set up the problem as an optimization problem with two constraints. This will allow us to reduce the inequality (\ref{b1}) to an extremal case, which we can characterize using Theorem \ref{necessary condition}.\ \vskip1mm 
Denote by $M(X)$, the $M$-functional, which is defined as $\sup_n \mathbb{P}(X=n)$. By letting $\alpha \to \infty$ in $N_{\alpha}(X)$, we have $$N_{\infty}(X) = M(X)^{-2}.$$ Therefore, inequality (\ref{b1}) is equivalent to 
\begin{equation}
\label{M-fun}
    M^2(X) (1 + \mathrm{Var}(X)) \leq 1.
\end{equation}

\begin{proof}[Proof of Theorem \ref{min-entropy}] Fix a log-concave random variable $X_0$. By approximation, one may assume that $X_0$ is compactly supported. Note that inequality (\ref{M-fun}) is invariant under translation. By definition, log-concavity is translation invariant as well. Therefore, we further assume that $X_0$ is supported on $\llbracket 0,L \rrbracket$, where $L \in \mathbb{Z}^{+}$. Denote by $\mathcal{P}(\llbracket 0,L \rrbracket)$, the set of log-concave probabilities supported on $\llbracket 0,L \rrbracket$. Let $h_1$ and $h_2$ be arbitrary functions defined on $\llbracket 0,L \rrbracket$ such that $h_1(n)=n^2 - \mathbb{E}[X_0^2]$ and $h_2(n)= \mathbb{E}[X_0]-n$ for all $n \in \llbracket 0,L \rrbracket $. Let $h = (h_1,h_2)$.\ \vskip1mm
Consider the set of all log-concave probability sequences supported on $\llbracket 0,L \rrbracket$ satisfying $\mathbb{E}[h_1(X)] \geq 0$ and $\mathbb{E}[h_2(X)] \geq 0$,
$$\mathcal{P}_h(\llbracket 0,L \rrbracket) = \{\mathbb{P}_X \in \mathcal{P}(\llbracket 0,L \rrbracket) \,:\, X \,\,\text{log-concave}\,,\, \mathbb{E}[h_1(X)]\geq 0 \,\text{and}\, \, \mathbb{E}[h_2(X)]\geq 0 \,\}\,.$$ 
Note that $\P_{X_0} \in \mathcal{P}_h(\llbracket 0,L \rrbracket)$ so ${P}_h(\llbracket 0,L \rrbracket)$ is non-empty. Denote by $\mathrm{conv} (\mathcal{P}_h(\llbracket 0,L \rrbracket))$, the convex hull of $\mathcal{P}_h(\llbracket 0,L \rrbracket)$. Suppose $\mathbb{P}_X$ is any extreme point in $\mathrm{conv} (\mathcal{P}_h(\llbracket 0,L \rrbracket))$. Let us assume that $\mathbb{P}_X$ is supported on $\llbracket \Tilde{K},\Tilde{L} \rrbracket \subseteq \llbracket 0,L \rrbracket$. Let $e^{-V}$ be the probability mass function of $\mathbb{P}_X$, where $V$ is convex on $\llbracket \Tilde{K},\Tilde{L}  \rrbracket$.\\\\
\textbf{STEP I:} The reduction of the inequality (\ref{M-fun}) to extreme points.\ \vskip1mm
This is essentially a consequence of the Krein-Milman (finite-dimensional) theorem, i.e. the supremum of any convex functional $\Phi$ over the set $\mathcal{P}_h(\llbracket 0, L \rrbracket)$ is attained at extreme points. Letting $\Phi(\mathbb{P}_X) = \sup_n \mathbb{P}(X=n)$, it suffices to prove the inequality (\ref{M-fun}) for $\mathbb{P}_X$- extreme.\\\\
\textbf{STEP II:} $e^{-V}$ has at most $3$ degrees of freedom. \vskip1mm
By contradiction, assume the degree of freedom of $e^{-V}$ is more than $3$. Then, there exist linear independent functions $W_1,W_2,W_3$ and $W_4$ defined on $\llbracket \Tilde{K},\Tilde{L}  \rrbracket$ and $\alpha>0$ such that for any $(\epsilon_1, \epsilon_2,\epsilon_3, \epsilon_4) \in [-\alpha,\alpha]^4$, the function $e^{-V} (1+ \sum_{i=1}^4 \epsilon_iW_i)$ is log-concave on $\llbracket \Tilde{K},\Tilde{L}  \rrbracket$. Let us consider the following set, which can be viewed as a $4-$dimensional cube centered at $\mathbb{P}_{X}$:
$$A=\{\,\text{log-concave functions of the form}\,\,e^{-V} (1+ \sum_{i=1}^4\epsilon_iW_i):\, |\epsilon_i| \leq \alpha\,\,\text{for all\,\,}i\}.$$
The set $A$ is not necessarily contained in $\mathcal{P}_h(\llbracket 0,L \rrbracket)$. Denote by $\Tilde{A}$, the following subset of $A$: 
\begin{equation*}
    \{\text{log-concave probability mass functions in $A$ satisfying \,\,} \mathbb{E}[h_1(X)]= 0\,\text{and}\,\, \mathbb{E}[h_2(X)]= 0 \}.
\end{equation*}
 $\Tilde{A} \neq \emptyset$ and $\Tilde{A}\subset \mathcal{P}_h(\llbracket 0,L\rrbracket)$. Consider the following system of linear equations:

\begin{align*}
    \epsilon_1 \sum W_1e^{-V}+\epsilon_2 \sum  W_2e^{-V}+ \epsilon_3 \sum  W_3e^{-V}+\epsilon_4 \sum W_4e^{-V} &= 0.\\
     \epsilon_1 \sum zW_1e^{-V}+\epsilon_2 \sum  zW_2e^{-V}+ \epsilon_3 \sum  zW_3e^{-V} + \epsilon_4 \sum zW_4e^{-V} &= 0.\\
     \epsilon_1 \sum z^2W_1e^{-V}+\epsilon_2 \sum  z^2W_2e^{-V}+ \epsilon_3 \sum  z^2W_3e^{-V} + \epsilon_4 \sum z^2W_4e^{-V} &= 0.
\end{align*}
When these equations hold, the associated probability mass functions lie in $\Tilde{A}$. Hence, the solutions of $(\epsilon_1,\epsilon_2,\epsilon_3,\epsilon_4)$  must lie on a line segment passing through
the origin. In particular, there exist $\mathbb{P}_{X_1}$
and $\mathbb{P}_{X_2}$ with mass functions of the form $e^{-V} (1+ \sum_{i=1}^4\epsilon_iW_i)$ and $e^{-V} (1 - \sum_{i=1}^4\epsilon_iW_i)$, respectively. But then, $\mathbb{P}_X = \frac{\mathbb{P}_{X_1} + \mathbb{P}_{X_2}}{2},\,$ a contradiction.\\\\
\textbf{STEP III:} Extreme points are log-affine or piecewise log-affine with 2 parts.\ \vskip1mm
It follows from Theorem \ref{necessary condition} that the mass function of any extremal is $e^{-V}$, where $V$ is affine or the maximum of $2$ affine functions on $\llbracket \Tilde{K},\Tilde{L}  \rrbracket$. If $X$ is log-affine (or piecewise log-affine) on $\llbracket \Tilde{K},\Tilde{L}  \rrbracket$, then $\Tilde{X} = X-\Tilde{K}$ is log-affine (or piecewise log-affine) on $\llbracket 0,M\rrbracket$, where $M = \Tilde{L}-\Tilde{K}$, and 
$$M^2(X)(1+\text{Var}(X)) = M^2(\Tilde{X})(1+\text{Var}(\Tilde{X})). $$
Therefore, it suffices to prove the desired inequality for $X$-log-concave with the mass function $e^{-V}$, where $V$ is affine or the maximum of $2$ affine functions on $\llbracket 0, M\rrbracket$. Note that if $V$ is affine, then $e^{-V}$ is monotone. Moreover, $e^{-V}$ is monotone if $V$ is non-increasing (or non-decreasing) with $2$ parts. Finally, if $e^{-V}$ is non-monotone, then it is of the following form
  \begin{equation*}
 \mathbb{P}(X=n)= C e^{-\alpha_1(N-n)} \mathbbm{1}_{\llbracket0,N \rrbracket}(n) + C e^{-\alpha_2(n-N)} \mathbbm{1}_{\llbracket N+1,M \rrbracket}(n), \end{equation*} where $C, \alpha_1, \alpha_2>0, N\geq 1$ and $M\geq2$. 
 \end{proof}
\section{Proof for extremals}
\label{S4}
\underline{\textit{CASE I}:} When $e^{-V}$ is monotone.\vskip2mm
We shall prove a slightly stronger result, i.e. we show that the inequality (\ref{M-fun}) holds for any non-negative monotone log-concave random variable (not necessarily having $3$ degree of freedom). Without loss of generality, assume that $X$ is non-increasing and supported on $\llbracket 0,M \rrbracket.$ Let $p$ be its probability mass function. Let $Z$ be geometric with the mass function $q(n) = \theta (1-\theta)^n,\, n=0,1,2,\dots$ Choose $\theta = \frac{1}{\mathbb{E}[X]+1}$ so that $\mathbb{E}[Z] = \mathbb{E}[X]$. It follows from [Theorem 2.6, \cite{MM2}] that $Z$ majorizes\footnote{A random variable $X$ is majorized by $Z$ in the convex order if $\mathbb{E}[\varphi(X)] \leq \mathbb{E}[\varphi(Z)]$ holds for any convex function $\varphi$.} $X$ in the convex order. In particular, we have $\mathbb{E}[X^2] \leq \mathbb{E}[Z^2]$. Equivalently, $\mathrm{Var}(X) \leq \mathrm{Var}(Z).$\vskip1mm
In fact, more is proven in \cite[Theorem 2.6]{MM2}, functions $p$ and $q$ cross exactly twice. However, since $\log(p)$ is concave and $\log(q)$ is affine, these functions crossing exactly twice necessarily imply that $p(0) \leq q(0)$. Since $p$ and $q$ are non-increasing, we have $\sup_n \mathbb{P}(X=n)=p(0)$ and $\sup_n \mathbb{P}(Z=n) = q(0)$, which imply
$$M(X)\leq M(Z).$$
For $Z$ geometric, it is easy to see that
$$ M^2(Z) (1 + \Var(Z)) \leq 1. $$
Therefore, combining the above, we conclude that
$$ M^2 (X) (1 + \Var(X))\leq M^2 (Z) (1 + \Var(Z))\leq 1. $$
\underline{\textit{CASE II}:} When $e^{-V}$ is non-monotone.\vskip2mm
We need to verify the inequality (\ref{M-fun}) for distributions of the following form $$\mathbb{P}(X=n)= C p_1^{-(N-n)} \mathbbm{1}_{\llbracket0, N \rrbracket}(n) + C p_2^{-(n-N)} \mathbbm{1}_{\llbracket N+1, N+K \rrbracket}(n),$$ where $p_1 = e^{\alpha_1}, p_2 = e^{\alpha_2}$ and $K = M-N.$ \ \vskip1mm

Observe $M(X) = C$. Since $\displaystyle \sum_{n=0}^{N+K} \mathbb{P}(X=n) = 1$, we have $$C \left( \sum_{n=0}^N p_1^{-n} + \sum_{n=1}^K p_2^{-n} \right) = 1.$$ Let $A = \displaystyle\sum_{n=0}^N p_1^{-n} $ and $B = \displaystyle \sum_{n=1}^K p_2^{-n}.$ Then, $C(A+B) = 1$. Let us compute the  variance of $X$. 
$$\mathbb{E}[X^2] = C \left(\sum_{n=0}^N \frac{n^2} {p_1^{N-n}} + \sum_{n=N+1}^{N+K} \frac{n^2}{p_2^{n-N}} \right)= C \left(\sum_{n=0}^N \frac{(N-n)^2}{p_1^n} + \sum_{n=1}^{K} \frac{(N+n)^2 }{p_2^n} \right).$$ 
After simplifying and substituting $A$ and $B$, the right-hand side is equal to,
$$C \left(  N^2 (A+B) + 2N \left( \sum_{n=1}^K \frac{n}{ p_2^n}- \sum_{n=1}^N \frac{n}{ p_1^n} \right ) +  \sum_{n=1}^K \frac{n^2}{ p_2^n}+ \sum_{n=1}^N \frac{n^2}{ p_1^n} \right).$$ By proceeding similarly, we get the following expression for $\mathbb{E}[X]^2.$ 
$$\mathbb{E}[X]^2 = C^2 \left((A+B)^2 N^2 + 2N (A+B) \left( \sum_{n=1}^K \frac{n}{ p_2^n}- \sum_{n=1}^N \frac{n}{ p_1^n} \right) + \left(\sum_{n=1}^K \frac{n}{ p_2^n}- \sum_{n=1}^N \frac{n }{p_1^n} \right)^2\,\right).$$ Combining $\mathbb{E}[X^2]$ and $\mathbb{E}[X]^2$, we get $$\mathrm{Var}(X) = \frac{1}{(A+B)^2}\left( (A+B) \left( \sum_{n=1}^K \frac{n^2}{ p_2^n}+ \sum_{n=1}^N \frac{n^2}{ p_1^n}\right)- \left(\sum_{n=1}^K \frac{n}{ p_2^n}- \sum_{n=1}^N \frac{n }{p_1^n} \right)^2\,\right).$$ After plugging the expressions for $M(X)$ and $\mathrm{Var}(X)$ into inequality (\ref{M-fun}) and rearranging the terms, we wish to prove for all integers $N, K\geq 1$ and real numbers $p_1,p_2\geq 1$,
\begin{equation}
\label{E4}
(A+B)^2 -1- \frac{1}{(A+B)} \left( \sum_{n=1}^K \frac{n^2}{ p_2^n}+ \sum_{n=1}^N \frac{n^2 }{p_1^n}\right) + \frac{1}{(A+B)^2} \left(\sum_{n=1}^K \frac{n}{p_2^n}- \sum_{n=1}^N \frac{n }{p_1^n} \right)^2 \geq 0.  
\end{equation}

\begin{remark}
The term $\frac{1}{(A+B)^2} \left(\sum_{n=1}^K \frac{n}{p_2^n}- \sum_{n=1}^N \frac{n }{p_1^n}\right)^2$ in the inequality may not be dropped. For example, take $p_1= \frac{5}{3},   p_2 = 500, N=9$ and $K=1$. Then
\begin{align*}
    (A+B)^3 - (A+B) - \left( \sum_{n=1}^K \frac{n^2}{ p_2^n}+ \sum_{n=1}^N \frac{n^2 }{p_1^n} \right)<0.
\end{align*}
\end{remark}
The rest of the paper is devoted to the proof of  
 inequality (\ref{E4}). In fact, we will prove a more general result. Let $x,y>0$. Let $S = \sum_{n=0}^N x^n$ and $T = \sum_{n=1}^K y^n.$ 
 
 We show for all $N,K\geq 1$,
\begin{align}
\label{E5}
(S+T)^4 - (S+T)^2 - (S+T)\left(\sum_{n=1}^N n^2\,x^n + \sum_{n=1}^K n^2\,y^n\right) + \left(\sum_{n=1}^N n\,x^n - \sum_{n=1}^K n\,y^n\right)^2 \geq 0.
\end{align}
After simplifying and rearranging the terms, the inequality (\ref{E5}) is equivalent to,
\begin{align*}
S^4 - S^2 - S \sum_{n=1}^N n^2\,x^n  & + \left( \sum_{n=1}^N n\,x^n\right)^2 \\
& + 2S^3\,T - T\sum_{n=1}^N n^2\,x^n- 2S\,T +3S^2\,T^2- 2\sum_{n=1}^N n\,x^n \,\sum_{n=1}^K n\,y^n \\
& + 2S\,T^3 - S\sum_{n=1}^K n^2\,y^n + 3S^2\,T^2 + S^3\,T\\
& + T^4 - T^2 + 2S\,T^3 + S^3\,T - T\sum_{n=1}^K n^2\,y^n + \left( \sum_{n=1}^K n\,y^n\right)^2 \geq 0.
\end{align*}
Therefore, it suffices to prove each of the following inequalities.

\begin{enumerate}[(I)]
 \item \label{I} $S^4 - S^2 - S \sum_{n=1}^N n^2\,x^n + \left( \sum_{n=1}^N n\,x^n\right)^2\geq 0.$
 \item \label{II} $2S^3\,T - T\sum_{n=1}^N n^2\,x^n- 2S\,T + 3S^2\,T^2 - 2\sum_{n=1}^N n\,x^n \,\sum_{n=1}^K n\,y^n \geq 0.$
\item \label{III} $2S\,T^3 - S\sum_{n=1}^K n^2\,y^n + 3S^2\,T^2 + S^3\,T  \geq 0.$ 
\item \label{IV} $T^4 - T^2 + 2S\,T^3 + S^3\,T - T\sum_{n=1}^K n^2\,y^n + \left( \sum_{n=1}^K n\,y^n\right)^2 \geq 0$.
\end{enumerate}
 We need the following lemma. The proof is elementary and included in the appendix.
 \begin{lemma} \label{poly}\  \begin{enumerate}[(a)]

 \item \label{a}  $S^2 = \displaystyle \sum_{n=0}^N (n+1)\,x^n + \displaystyle \sum_{n=N+1}^{2N} (2N-n+1)\,x^n.$
 
 \item \label{b} $S^3\geq \dfrac{1}{2}\, \displaystyle \sum_{n=0}^{N} (n+1)(n+2)\,x^n + \dfrac{1}{2}  \sum_{n=N+1}^{2N} \big[(3N-n+1)(n-N) + (n+2)(2N-n+1)\big]\,x^n$.

 \item \label{c} $T^2 = \displaystyle \sum_{n=2}^{K+1} (n-1)\,y^n + \sum_{n=K+2}^{2K} (2K - n+1)\, y^n.$

 \item \label{d} $T^3 \geq \displaystyle \dfrac{1}{2}\, \sum_{n=3}^{K+2} (n-1)(n-2)\,y^n + \dfrac{1}{2} \sum_{n=K+3}^{2K+1} \big[(n-K-2)(3K-n+1)+(n-1)(2K-n+2)\big]\,y^n$.

\item \label{e} $T^4 \geq \dfrac{1}{6}\, \displaystyle \sum_{n=4}^{K+3} (n-1)(n-2)(n-3)\,y^n + \displaystyle \sum_{n={K+4}}^{2K+2} e_{K}(n)\,y^n,$ \\
where $e_K (n) = \dfrac{1}{6}\, \big[(2K-n+3)(n^2 + 2nK-3n-2K^2-6K+2) + 2(n-K-3)(n-K-2)(4K-n+1) \big]$.

 \item \label{f}$ T \displaystyle\sum_{n=1}^K n^2y^n  = \frac{1}{6}\, \displaystyle \sum _{n=2}^{K+1} n(n-1)(2n-1)\,y^n + \displaystyle \sum_{n=K+2}^{2K} f_{K}(n) \,y^n$, \\
 where $f_{K}(n)= \dfrac{1}{6}\, (2K-n+1)(2n^2- n(2K+1)+2K(K+1))$.
 
  \item \label{g} $ \left(\displaystyle \sum_{n=1}^K n\,y^n\right)^2= \dfrac{1}{6} \,\displaystyle \sum _{n=2}^{K+1} n(n-1)(n+1)\,y^n +\displaystyle \sum _{n=K+2}^{2K} g_K(n)\, y^n$, \\
  where $g_K(n) = \dfrac{1}{6}\,(2K-n+1)(n^2 + n(2K +1) -2K(K+1))$.\
\end{enumerate}
\end{lemma}
\vskip2mm
\begin{proof}[Proof of inequality (\ref{I})] By (\ref{a}) in Lemma \ref{poly},
\begin{align*}
    S^4 = (S^2)^2 \geq \left(\sum_{n=0}^N (n+1)\,x^n\right)^2 \geq \left(\displaystyle \sum_{n=1}^N n\,x^n \right)^2 + \left(\displaystyle \sum_{n=0}^N x^n \right)^2 + \sum_{n=0}^N x^n \sum_{n=1}^N n\,x^n.
\end{align*}
Therefore,
 \begin{align*}
 S^4 - S^2 - S \sum_{n=1}^N n^2\,x^n + \left( \sum_{n=1}^N n\,x^n\right)^2 & \geq  2 \left(\displaystyle \sum_{n=1}^N n\,x^n \right)^2 +  \sum_{n=0}^N x^n \sum_{n=1}^N n\,x^n - \sum_{n=0}^N x^n\sum_{n=1}^N n^2\,x^n \\
 & = 2 \left(\displaystyle \sum_{n=1}^N n\,x^n \right)^2  + \sum_{n=0}^N x^n \left(\sum_{n=1}^N n(1-n)\,x^n\right) \\
 & = \sum_{n=0}^{2N} c_n\, x^n,
  \end{align*}
where $c_n = \displaystyle \sum_{i+j=n} 2ij + i (1-i)$. Note that $n$ ranges from $0$ to $2N$ while $0\leq i,j \leq N$. We conclude since,
\begin{equation*}
c_n= \begin{cases}
\displaystyle \sum_{i=0}^n 2i(n-i)+ i(1-i) = 0, & 0\leq n \leq N\\
\displaystyle \sum_{i=n-N}^N 2i(n-i)+ i(1-i) = (N+1)(n-N)(2N+1-n), & N<n\leq 2N
\end{cases}
\end{equation*}
\end{proof}
  
\begin{proof}[Proof of inequality (\ref{II})]
By (\ref{b}) in Lemma \ref{poly}
\begin{equation}
\begin{aligned}
  2S^3\,T - T\sum_{n=1}^N n^2\,x^n- 2S\,T &=  T \left(2S^3 - \sum_{n=1}^N n^2\,x^n -2S \right)\\
  &\geq  T \left(\sum_{n=0}^N (n+1)(n+2)\,x^n - \sum_{n=1}^N n^2\,x^n - 2 \sum_{n=0}^N x^n\right)\\
    &= 3 T \sum_{n=1}^N n\,x^n.
\end{aligned}
\end{equation}
Together with (6), we can easily verify the desired inequality for $K=1$. Assume $K\geq 2$. By (\ref{a}) and (\ref{c}) in Lemma \ref{poly},
\begin{equation}
\begin{aligned}
 3S^2\,T^2 - 2\sum_{n=1}^N n\,x^n \,\sum_{n=1}^K n\,y^n & \geq 3\sum_{n=0}^N (n+1)\,x^n \sum_{m=2}^{K+1} (m-1)\,y^m   - 2\sum_{n=1}^N n\,x^n \,\sum_{m=1}^K m\,y^m\\ 
 &> \sum_{{0\leq n\leq N} \atop {2 \leq m \leq K}} \left(3(n+1)(m-1) -2nm \right) \,x^n\,y^m -2y \sum_{n=1}^N n\,x^n \\
& > -3\sum_{{0\leq n\leq N} \atop {2 \leq m \leq K}}n\,x^n\,y^m -2y \sum_{n=1}^N n\,x^n\\
&= -3 \sum_{m=2}^K y^m \sum_{n=1}^N n\,x^n -2y \sum_{n=1}^N n\,x^n > -3T \sum_{n=1}^N n\,x^n.
\end{aligned}
\end{equation}
We conclude by combining (6) and (7).

\end{proof}

\begin{proof}[Proof of inequality (\ref{III})]
Since $S\geq 1,$
\begin{align*}
2S\,T^3 - S\sum_{n=1}^K n^2y^n + 3S^2\,T^2 + S^3\,T \geq S \left(2T^3-\sum_{n=1}^K n^2\,y^n + 3T^2 + T\right).
 \end{align*}
It is enough to show $2T^3-\sum_{n=1}^K n^2\,y^n + 3T^2 + T>0$. This is immediate when $K=1,2$. Suppose $K\geq 3$. By (\ref{c}) and (\ref{d}) in Lemma \ref{poly},
\begin{align*}
2T^3-\sum_{n=1}^K n^2y^n + 3T^2 + T &\geq \sum_{n=3}^{K+2} (n-1)(n-2)\,y^n - \sum_{n=1}^K n^2\,y^n  + 3\sum_{n=2}^{K+1} (n-1)\,y^n + \sum_{n=1}^K y^n\\
 &> \sum_{n=3}^K \left( (n-1)(n-2)-n^2 + 3(n-1)+1 \right)\,y^n=0.
 \end{align*}
\end{proof}   
    
\begin{proof}[Proof of inequality (\ref{IV})] Note that
\begin{align*}
T^4 - T^2 + 2S\,T^3 &+ S^3\,T - T\sum_{n=1}^K n^2y^n +  \left( \sum_{n=1}^K n\,y^n\right)^2 \\  
& \geq T^4 - T^2 + 2T^3 + T - T\sum_{n=1}^K n^2\,y^n + \left( \sum_{n=1}^K n\,y^n\right)^2.
\end{align*}
We will show 
\begin{align*}
    P_K(y) \coloneqq T^4 - T^2 + 2T^3 + T - T\sum_{n=1}^K n^2\,y^n + \left( \sum_{n=1}^K n\,y^n\right)^2\geq 0.
\end{align*}

The cases $K=1,2$ can be checked directly. Let us assume $K\geq 3$. $P_K(y)$ is a polynomial of degree $4K$, say of the form $\sum_{n=0}^{4K} d_n y^n$. We will show $d_n\geq 0$ for $0\leq n \leq 4K$. Due to $T^4$ and $T^3$ terms, necessarily $d_n>0$ for $2K+1 \leq n \leq 4K$. It remains to show $d_n \geq 0$ for $0\leq n \leq 2K$. First, observe the following: $d_0=0$, $d_1=1$, and $d_2=d_3=0$. In order to treat the remaining cases, we apply (\ref{c})-(\ref{g}) in Lemma \ref{poly} to $d_n.$ \\\\
\textit{When $4\leq n \leq K$}: 
$d_n \geq \frac{(n-1)(n-2)(n-3)}{6}\, - \,(n-1)\, + \,(n-1)(n-2)\, + 1 \,-\, \,\frac{n(n-1)(2n-1)}{6} + \frac{n(n-1)(n+1)}{6} =\frac{(n-2)(n-3)}{2}>0$.\\\\
\textit{When $n= K+1$}: $d_n\geq \frac{K(K-1)(K-2)}{6}\,-\,K\,+\,K(K-1)\,-\, \frac{K(K+1)(2K+1)}{6}+ \,\frac{K(K+1)(K+2)}{6} = \frac{K(K-3)}{2}\geq 0$.\\\\
\textit{When $n= K+2$}:
$d_n \geq \frac{K(K-1)(K+1)}{6} - \,(K-1)\, + \,K(K+1)\, - \frac{(K-1)(2K^2+5K+6)}{6}\,+ \,\frac{(K-1)(K+1)(K+6)}{6} = \frac{3K^2-K+2}{2}>0$.\\\\
\textit{When $n= K+3$}:
$d_n \geq \frac{K(K+1)(K+2)}{6}\, - \,(K-2) \,+\, (K-1)(K+4)\,- \frac{(K-2)(2K^2 +7K+15)}{6} +\frac{(K-2)(K^2 +11K+12)}{6} = \frac{5K^2+K-2}{2}>0$.\\\\
\textit{When $K+4\leq n \leq 2K$}: After applying Lemma \ref{poly}, simplifying and rearranging the terms, we get

\begin{align*}
    d_n \geq -\frac{1}{3}n^3 + \left(K+ \frac{1}{2} \right)n^2 +\frac{11}{6}n -\frac{2K^3}{3}-K^2 - \frac{13K}{3}  -2.
\end{align*}
Let \begin{equation*}
    V_K(n) = -\frac{1}{3}n^3 + \left(K+ \frac{1}{2} \right)n^2 +\frac{11}{6}n -\frac{2K^3}{3}-K^2 - \frac{13K}{3}  -2.
\end{equation*}
We will show that $V_K \geq 0$ for $K+4\leq n \leq 2K$. Taking the derivative of $V_K$ w.r.t $n$ yields
\begin{equation*}
    \dfrac{dV_K}{dn} = - n^2 + n (2K+1) + \dfrac{11}{6}.
\end{equation*}
$\dfrac{dV_K}{dn}\geq 0$ iff  \begin{equation*}
\dfrac{1}{6} \left(-\sqrt{3} \sqrt{12K^2 + 12K + 25}+6K +3\right) \leq n \leq \dfrac{1}{6} \left(\sqrt{3} \sqrt{12K^2 + 12K + 25}+6K +3\right). \end{equation*}
But, 
\begin{equation*}
    \left[\dfrac{1}{6} \left(-\sqrt{3} \sqrt{12K^2 + 12K + 25}+6K +3\right), \dfrac{1}{6} \left(\sqrt{3} \sqrt{12K^2 + 12K + 25}+6K +3\right) \right] \supseteq [K,2K]
\end{equation*}
for all $K\geq 0.$\ \vskip1mm
Therefore $\dfrac{dV_K}{dn}\geq 0$ for all $K+4\leq n \leq 2K$. Moreover, $V_K(K+4) = \frac{K(7K+3)}{2}-8>0$ for all $K\geq 3$, completing the proof.

\end{proof}

\begin{appendix}
    \section*{Appendix: Some polynomial identities}
Each identity in Lemma \ref{poly} can be obtained by the standard product of polynomials. The desired coefficients are given by convolution operation.
\begin{proof}[Proof of Lemma \ref{poly}]\
\label{polynomial identities}
 \begin{enumerate}[(a)]
\item \begin{align*}
S^2  & =  \sum_{n=0}^{2N}\left(\sum_{{i+j = n}\atop {0\leq i,j\leq N}} 1\right)x^n \\
 & = \sum_{n=0}^N \left(\sum_{i=0}^n 1\right)x^n + \sum_{n=N+1}^{2N} \left(\sum_{i=n-N}^{N} 1 \right)x^n\\
 &= \sum_{n=0}^N (n+1)\, x^n + \sum_{n=N+1}^{2N} (2N-n+1)\,x^n.
\end{align*}

\item \begin{align*}
S^3=S \cdot S^2= \sum_{n=0}^N x^n \left( \sum_{n=0}^N (n+1)x^n + \sum_{n=N+1}^{2N} (2N-n+1)x^n\right).
\end{align*}
Let us consider the first summation.
\begin{align*}
 \sum_{n=0}^N x^n  \sum_{n=0}^N (n+1)\,x^n &= \sum_{n=0}^{2N} \left(\sum_{{i+j = n}\atop {0\leq i,j\leq N}}(i+1)\right)x^n\\
& = \sum_{n=0}^N \left(\sum_{i=0}^n (i+1) \right)x^n + \sum_{n=N+1}^{2N} \left(\sum_{i=n-N}^N (i+1) \right)x^n\\
& = \frac{1}{2}\, \sum_{n=0}^N (n+1)(n+2)\,x^n + \frac{1}{2} \sum_{n=N+1}^{2N} (n+2)(2N-n+1)\, x^n. 
\end{align*}
Similarly, for the second summation
\begin{align*}
 \sum_{n=0}^N x^n  \sum_{n=N+1}^{2N} (2N- n+1)\,x^n &= x^{N+1} \sum_{n=0}^N x^n \sum_{n=0}^{N-1} (N-n)\,x^n \\
 & \geq x^{N+1} \sum_{n=0}^{N-1} x^n \sum_{n=0}^{N-1} (N-n)\,x^n   \\
  & \geq  x^{N+1} \sum_{n=0}^{2N-2} \left(\sum_{{i+j = n}\atop {0\leq i,j\leq N-1}} (N-i)\right)x^n \\
&\geq x^{N+1} \displaystyle \sum_{n=0}^{N-1} \left( \sum_{i=0}^n (N-i)\right) x^n\\
&= x^{N+1} \cdot \frac{1}{2}\sum_{n=0}^{N-1} (n+1)(2N-n)\,x^n\\
&= \frac{1}{2}\, \sum_{n=N+1}^{2N} (n-N)(3N-n+1)\,x^n.
\end{align*}
We conclude by combining the two terms.
\item Since $T^2 = y^2 \left(\displaystyle \sum_{n=0}^{K-1} y^n\right)^2$, apply (a) with $N=K-1$ and $x=y$.

\item Since $T^3 = y^3 \left(\displaystyle \sum_{n=0}^{K-1} y^n\right)^2$, apply (b) with $N=K-1$ and $x=y$.
\item First, we consider $S^4$. By (a), 
\begin{align*}
    S^4 \geq \left(\sum_{n=0}^N (n+1)\,x^n\right)^2 + 2\,\sum_{n=0}^N (n+1)\,x^n \sum_{n=N+1}^{2N} (2N-n+1)\,x^n
\end{align*}
Consider the two summations on the right-hand side separately.
\begin{align*}
\left( \sum_{n=0}^N (n+1)\,x^n\right)^2 &=\sum_{n=0}^{2N} \left(\sum_{{i+j = n}\atop {0\leq i,j\leq N}}(i+1)(j+1)\right)x^n \\ 
&= \sum_{n=0}^{N} \left(\sum_{i=0}^n (i+1)(n-i+1)\right)x^n + \sum_{n=N+1}^{2N} \left(\sum_{i=n-N}^N (i+1)(n-i+1)\right)x^n\\
&= \frac{1}{6}\, \sum_{n=0}^{N} (n+1)(n+2)(n+3)\,x^n  \\ & + \frac{1}{6} \sum_{n=N+1}^{2N} \big[(2N-n+1)(n^2 + n(2N+7) - 2(N^2 + N-3)) \big]\,x^n.
\end{align*}
By proceeding as before, for the second summation, we have
\begin{align*}
\sum_{n=0}^N (n+1)\,x^n \sum_{n=N+1}^{2N} (2N-n+1)\,x^n &= x^{N+1} \sum_{n=0}^N (n+1)\,x^n \sum_{n=0}^{N-1}(N-n) \,x^n  \\
& \geq   x^{N+1} \sum_{n=0}^{N-1} (n+1)\,x^n \sum_{n=0}^{N-1}(N-n)\, x^n  \\
&= x^{N+1} \sum_{n=0}^{2N-2} \left(\sum_{{i+j = n}\atop {0\leq i,j\leq N-1}}(N-i)(j+1)\right)\,x^n\\
&\geq x^{N+1}\sum_{n=0}^{N-1} \left(\sum_{i=0}^n(N-i)(n-i+1)\right)x^n\\
&= x^{N+1} \cdot \frac{1}{6} \sum_{n=0}^{N-1} (n+1)(n+2)(3N-n)\, x^n\\
&= \frac{1}{6} \sum_{n=N+1}^{2N} (n-N)(n-N+1)(4N-n+1)\,x^n.
\end{align*}
After combining the two identities, we get
\begin{equation*}
  S^4  \geq \frac{1}{6} \,\displaystyle \sum_{n=0}^{N} (n+1)(n+2)(n+3)\,x^n + \displaystyle \sum_{n=N+1}^{2N} \Tilde{e}_K (n)\, x^n,
  \end{equation*}
  
where $\Tilde{e}_K (n) = \dfrac{1}{6}\, \big[(2N-n+1)(n^2 + n(2N+7) - 2(N^2 + N-3)) + 2(n-N)(n-N+1)(4N-n+1) \big]$.\\\\
Since $T^4 = y^4 \left(\displaystyle \sum_{n=0}^{K-1} y^n \right)^4,$ by letting $N= K-1$ and $x=y$ in $S^4$, we can deduce the desired inequality.
\item 
\begin{align*}
T \displaystyle \sum_{n=1}^K n^2y^n = \displaystyle \sum_{n=1}^K y^n \displaystyle \sum_{n=1}^K n^2\,y^n &= \sum_{n=2}^{2K} \left(\sum_{{i+j = n}\atop {0\leq i,j\leq K}} i^2 \right)\,y^n\\
&= \sum_{n=2}^{K+1} \left(\sum_{i=1}^{n-1} i^2\right)\,y^n + \sum_{n=K+2}^{2K} \left(\sum_{i=n-K}^{K}i^2\right)\,y^n \\
& = \frac{1}{6} \sum _{n=2}^{K+1} n(n-1)(2n-1)\,y^n \\ &+ \frac{1}{6}\sum_{n=K+2}^{2K} (2K-n+1)(2n^2- n(2K+1)+2K(K+1))\,y^n.
\end{align*}
\item 
\begin{align*}
    \left(\sum_{n=1}^K n\,y^n\right)^2  & = \sum_{n=2}^{2K} \left(\sum_{{i+j = n}\atop {0\leq i,j\leq K}}ij\right)\,y^n \\
   &= \sum_{n=2}^{K+1} \left( \sum_{i=1}^{n-1} i(n-i)\right)\,y^n + \sum_{n=K+2}^{2K} \left( \sum_{i=n-K}^{K} i(n-i)\right)\,y^n\\
   & =  \frac{1}{6}\, \sum_{n=2}^{K+1} n(n-1)(n+1)\,y^n \\ 
   & + \frac{1}{6} \sum_{n=K+2}^{2K} (2K-n+1)(n^2 + n(2K+1) -2K(K+1))\, y^n.
\end{align*}

\end{enumerate}
\end{proof}

\end{appendix}


\vskip2cm


\noindent Heshan Aravinda \\
Department of Mathematics \\
University of Florida \\
Gainesville, FL 32611, USA \\
heshanaravinda.p@ufl.edu
\end{document}